 \documentclass{amsart}

\usepackage{amsmath,amsthm,amsfonts, amssymb,amscd}
\usepackage[all]{xy}

\newtheorem{theorem}{Theorem}[section]
\newtheorem{lemma}[theorem]{Lemma}
\newtheorem{example}[theorem]{Example}
\newtheorem{remark}[theorem]{Remark}
\newtheorem{corollary}[theorem]{Corollary}
\newtheorem{proposition}[theorem]{Proposition}

\newcommand{\lex}{\,\overrightarrow{\times}\,}
\newcommand{\dx}{\, \mbox{\rm d}}

\newcommand{\Aff}{\mbox{\rm Aff}}

\begin{document}
\title{Every  State on Interval Effect Algebra  is  Integral}
\author{Anatolij Dvure\v censkij}
\date{}
\maketitle

\begin{center}
\footnote{


The  author thanks  for the support by Center of Excellence SAS
-~Quantum Technologies~-,  ERDF OP R\&D Projects CE QUTE ITMS
26240120009 and meta-QUTE ITMS 26240120022, the grant VEGA No.
2/0032/09 SAV and by the Slovak Research and Development Agency
under the contract No. APVV-0071-06, Bratislava. }
\small{Mathematical Institute,  Slovak Academy of Sciences\\
\v Stef\'anikova 49, SK-814 73 Bratislava, Slovakia\\
E-mail: {\tt dvurecen@mat.savba.sk}, }
\end{center}

\begin{abstract}  We show that every state on an interval effect algebra
is an integral through some regular Borel probability measure
defined on the Borel $\sigma$-algebra of a compact Hausdorff
simplex. This is true for every effect algebra satisfying (RDP) or
for every  MV-algebra. In addition, we show that each state on an
effect subalgebra of an interval effect algebra $E$ can be extended
to a state on $E.$ Our method represents also every state on the set
of  effect operators of a Hilbert space as an integral.
\end{abstract}

\section{Introduction} 

Effect algebras were defined by Foulis and Bennett in \cite{FoBe} in
order to axiomatize some measurements in quantum mechanics that are
modeled by POV-measures as partial algebras with a partially defined
addition, $+,$ an analogue of join of two orthogonal elements.
Effect algebras generalize Boolean algebras, orthomodular lattices
and posets, orthoalgebras as well as MV-algebras. Because Hilbert
space quantum mechanics ``lives" in a Hilbert space, $H,$ the system
of all Hermitian operators, $\mathcal B(H),$ the system of all
effect operators, $\mathcal E(H)$, i.e. the set of all Hermitian
operators between the zero operator and the identity, as well as the
system of all orthogonal projectors, $\mathcal P(H),$ are crucial
algebraic structures: $\mathcal B(H)$ is a po-group that is an
antilattice, [LuZa, Thm 58.4], i.e. only comparable elements have
join, $\mathcal E(H)$ is an effect algebra, and $\mathcal P(H)$ is a
complete orthomodular lattice. They are fundamental examples of
quantum structures; more on quantum structures see \cite{DvPu}.

We note that effect algebras are equivalent to D-posets defined by
K\^opka and Chovanec \cite{KoCh}, where a primary notion is
difference of comparable events.

An important subfamily of effect algebras are MV-algebras introduced
by Chang \cite{Cha} that model many-valued \L ukasiewicz logic. The
role of MV-algebras for effect algebras is analogous as that of
Boolean algebras for orthomodular posets: any orthomodular poset or
a lattice effect algebra can be covered by a system of mutually
compatible substructures that in the first case are Boolean
subalgebras whilst in the second case are MV-subalgebras.

Many useful effect algebras are intervals in some unital po-group
groups. Such examples are effect algebras with the Riesz
Decomposition Property (RDP for short), or unigroups or MV-algebras.

A state is an analogue of a probability measure for quantum
structures. Roughly speaking, it is a positive functional $s$ on an
effect algebra $E$ that preserves $+$ and $s(1)=1$ where $1$ is the
greatest element in $E.$  In general, an effect algebra may have no
state, but effect algebra with (RDP) or an interval one has at least
one state.

The system of all states is always a convex compact Hausdorff space
whose the set of extremal states is not necessarily compact. Anyway,
using some analysis of compact convex sets it is possible to show
that if $E$ is an interval effect algebra or even if it satisfies
(RDP) or an MV-algebra, then a state is an integral through a
regular Borel probability measure defined on the Borel
$\sigma$-algebra over a compact Hausdorff simplex. This result
generalizes the analogous results by \cite{Kro, Kro1} and Panti
\cite{Pan} proved for states on MV-algebras. We show also conditions
when, for every state $s$ on $E,$ there is a unique regular Borel
probability measure.

In addition, we generalize Horn--Tarski Theorem that says that every
state on a Boolean subalgebra of a Boolean algebra $B$ can be
extended to a state on $B.$ We show that every state on an effect
subalgebra of an interval effect algebra $E$ can be extended to a
state on $E.$

The article is organized as follows. Section 2 presents elements of
effect algebras, and states on effect algebras are recalled in
Section 3. Section 4 is devoted to affine continuous functions on a
convex compact Hausdorff topological space and Section 5 shows
importance of Choquet and Bauer simplices for our study.  Finally,
the main results are formulated and proved in Section 6. In
Conclusion, we summarize our results.

\section{Elements of Effect Algebras}

An {\it effect algebra} is by \cite{FoBe} a partial algebra $E =
(E;+,0,1)$ with a partially defined operation $+$ and two constant
elements $0$ and $1$  such that, for all $a,b,c \in E$,
\begin{enumerate}

\item[(i)] $a+b$ is defined in $E$ if and only if $b+a$ is defined, and in
such a case $a+b = b+a;$

 \item[(ii)] $a+b$ and $(a+b)+c$ are defined if and
only if $b+c$ and $a+(b+c)$ are defined, and in such a case $(a+b)+c
= a+(b+c);$

 \item[(iii)] for any $a \in E$, there exists a unique
element $a' \in E$ such that $a+a'=1;$

 \item[(iv)] if $a+1$ is defined in $E$, then $a=0.$
\end{enumerate}

We recommend \cite{DvPu} for more details on effect algebras.

If we define $a \le b$ if and only if there exists an element $c \in
E$ such that $a+c = b$, then $\le$ is a partial order on $E$, and we
denote $c:=b-a.$ Then $a' = 1 - a$ for any $a \in E.$

We will assume that $E$ is not degenerate, i.e. $0\ne 1$.

Let $E$ and $F$ be two effect algebras. A mapping $h:\ E \to F$ is
said to be a {\it homomorphism} if (i) $h(a+b) = h(a) + h(b)$
whenever $a+b$ is defined in $E$, and (ii) $h(1) =1$. A bijective
homomorphism $h$ such that $h^{-1}$ is a homomorphism is said to be
an {\it isomorphism} of $E$ and $F$.

Let ${\mathcal E}$ be a system of $[0,1]$-valued functions on
$\Omega\ne \emptyset$ such that (i) $1 \in {\mathcal E}$, (ii) $f
\in {\mathcal E}$ implies $1-f \in {\mathcal E}$, and (iii) if $f,g
\in {\mathcal E}$ and $f(\omega) \le 1 -g(\omega)$ for any $\omega
\in \Omega$, then $f+g \in {\mathcal E}$. Then ${\mathcal E}$ is an
effect algebra of $[0,1]$-valued functions, called an {\it
effect-clan}, that is not necessarily a Boolean algebra nor an
MV-algebra.

We say that an effect algebra  $E$ satisfies  the  {\it Riesz
Decomposition Property}, (RDP) for short, if $x \le y_1 + y_2$
implies that there exist two elements $x_1, x_2 \in E$ with $x_1 \le
y_1 $ and $x_2 \le y_2$ such that $x = x_1 + x_2$.

We recall that  $E$ has (RDP) if and only if, \cite[Lem.
1.7.5]{DvPu}, $x_1 + x_2 = y_1 + y_2$ implies there exist four
elements $c_{11}, c_{12}, c_{21}, c_{22} \in E$ such that $x_1 =
c_{11} + c_{12},$ $x_2 = c_{21} + c_{22},$ $y_1 = c_{11} + c_{21},$
and $y_2 = c_{12} + c_{22}.$

A partially ordered Abelian group $(G;+,0)$ (po-group in short) is
said to satisfy  {\it interpolation} provided  given
$x_1,x_2,y_1,y_2$ in $G$ such that $x_1,x_2\leq y_1,y_2$  there
exists $z$ in $G$ such that $x_1,x_2\leq z\leq y_1,y_2.$  We say
also that $G$ is an {\it interpolation group.}

A po-group $(G,u)$ with strong unit $u$ (= order unit, i.e. given $g
\in G,$ there is an integer $n\ge 1$ such that $g\le n$) is said to
be a {\it unital po-group}. If the set
$$
\Gamma(G,u) :=\{g \in G:\, 0 \le g \le u\} \eqno(2.1)
$$
is endowed with the restriction of the group addition $+$, then
$(\Gamma(G,u); +,0,u)$ is an effect algebra.

Ravindran \cite{Rav} (\cite[Thm 1.7.17]{DvPu}) showed that every
effect algebra with (RDP) is of the form (2.1) for some
interpolation unital po-group $(G,u)$:

\begin{theorem}\label{th:2.1}  Let $E$ be an effect algebra with
{\rm (RDP)}. Then there exists a unique (up to isomorphism of
Abelian po-groups with strong unit)  interpolation group $(G,u)$
with strong unit such that $\Gamma(G,u)$ is isomorphic with $E$.
\end{theorem}

Moreover, the category of effect algebras with {\rm (RDP)} is
categorically equivalent with the category of unital po-groups with
interpolation; the categorical equivalence is given by the functor:
$\Gamma: (G,u) \mapsto \Gamma(G,u).$

An effect algebra of the form $\Gamma(G,u)$ for some element $u\ge
0$ or if it is isomorphic with $\Gamma(G,u)$  is called also an {\it
interval effect algebra}. For example,  the system, ${\mathcal
E}(H),$ of all Hermitian operators, $A,$ of a Hilbert space $H$
(real, complex or quaternionic) such that $O \le A\le I,$ where the
ordering of Hermitian operators is defined by the property: $A \le
B$ iff $(A\phi,\phi) \le (B\phi,\phi)$ for all $\phi \in H,$ is an
interval effect algebra such that  ${\mathcal E}(H)
=\Gamma({\mathcal B}(H),I).$ We note that $\mathcal E(H)$ is without
(RDP).

Let $u$ be a positive element of an Abelian po-group $G.$ The
element $u$ is said to be {\it generative} if every element $g \in
G^+$ is a group sum of finitely many elements of $\Gamma(G,u),$ and
$G = G^+-G^+.$ Such an element is  a strong unit \cite[Lem
1.4.6]{DvPu} for $G$ and the couple $(G,u)$ is said to be a {\it
po-group with generative strong unit}.  For example, if $u$ is a
strong unit of an interpolation po-group $G,$ then $u$ is
generative.  The same is true for $I$ and $\Gamma({\mathcal
B}(H),I).$

Let $E$ be an effect algebra and $H$ be an Abelian (po-) group. A
mapping $p: G \to H$ that preserves $+$ is called an $H$-{\it valued
measure} on $E.$

\begin{remark}\label{re:amb} {\rm
If $E$ is an interval effect algebra, then there is a po-group $G$
with a generative strong unit $u$ such that $E \cong \Gamma(G,u)$
and every $H$-valued measure $p:\Gamma(G,u) \to H$ can be extended
to a group-homomorphism $\phi$ from $G$ into $H.$ If $H$ is a
po-group, then $\phi$ is a po-group-homomorphism. Then $\phi$ is
unique and $(G,u)$ is also unique up to isomorphism of unital (po-)
groups, see \cite[Cor 1.4.21]{DvPu}; the element $u$ is said to be a
{\it universal strong unit} for $\Gamma(G,u)$ and the couple $(G,u)$
is said to be a {\it unigroup}. In particular, the identity operator
$I$ is a universal strong unit for $\Gamma({\mathcal B}(H),I),$
\cite[Cor 1.4.25]{DvPu}, similarly for $\Gamma(A,I),$ where $A$ is a
von Neumann algebra on $H.$}
\end{remark}

Also in this case, the category of interval effect algebras is
categorically equivalent with the category of unigroups $(G,u)$ with
the functor $\Gamma: (G,u)\mapsto \Gamma(G,u).$

We notify that also a non-commutative version of effect algebras,
called {\it pseudo effect algebras}, was recently introduced in
\cite{DvVe1,DvVe2}; in some important cases they are also intervals
but in unital po-groups that are not necessarily Abelian.

A very important family of effect algebras is the family of
MV-algebras, which were introduced by Chang \cite{Cha}:

An {\it MV-algebra} is an algebra $(M;\oplus,^*,0)$ of signature
$\langle 2,1,0\rangle,$ where $(M;\oplus,0)$ is a commutative monoid
with neutral element $0$, and for all $x,y \in M$
\begin{enumerate}
\item[(i)]  $(x^*)^*=x,$
\item[(ii)] $x\oplus 1 = 1,$ where $1=0^*,$
\item[(iii)] $x\oplus (x\oplus y^*)^* = y\oplus (y\oplus x^*)^*.$
\end{enumerate}

We define also an  additional total operation $\odot$ on $M$ via
$x\odot y:= (x^*\oplus y^*)^*.$  For more on MV-algebras, see the
monograph \cite{CDM}.

If we define a partial operation $+$ on an MV-algebra $M$ in such a
way that $a+b$ is defined in $M$ if and only if $a \le b^*$ (or
equivalently $a\odot b=0$) and we set $a+b:=a\oplus b$, then
$(M;+,0,1)$ is a lattice ordered effect algebra with (RDP), where
$a':=a^*.$ According to \cite{Mun}, every MV-algebra is isomorphic
with $(\Gamma(G,u); \oplus,^*, 0)$, where $(G,u)$ is an $\ell$-group
(= lattice ordered group) with strong unit and
$$
a\oplus b =(a+b)
\wedge u,\quad a^* = u-a,\quad a,b \in \Gamma(G,u).
$$

It is worthy recalling   that a lattice ordered effect algebra $E$
with (RDP) can be converted into an MV-algebra (in such a way that
both original $+$ and the partial addition induced  from the
MV-structure coincide); we define $x\oplus y:= x+ (y\wedge(x\wedge
y)')$ and $x^*=x',$ see \cite[Thm 1.8.12]{DvPu}.

\section{States on Effect Algebras}

We present some notions and properties about states on effect
algebras. For further reading we recommend \cite{DvPu, Dvu}.

A {\it state} on an effect algebra $E$ is any mapping $s: \ E \to
[0,1]$ such that (i) $s(1) = 1$, and (ii) $s(a+b) = s(a) + s(b)$
whenever $a+b$ is defined in $E$.  We denote by ${\mathcal S}(E)$
the set of all states on $E$. It can happen that ${\mathcal S}(E)$
is empty, see e.g. \cite[Exam. 4.2.4]{DvPu}. In important cases, for
example, when $E$ satisfies (RDP) or, more generally, if $E$ is an
interval effect, ${\mathcal S}(E)$ is nonempty, see below.
${\mathcal S}(E)$ is always a convex set, i.e. if $s_1, s_2$ are
states on an effect algebra $E$ and $\lambda \in [0,1],$ then $s =
\lambda s_1 + (1-\lambda)s_2$ is a also a state on $E.$ A state $s$
is said to be {\it extremal} if $s = \lambda s_1 + (1-\lambda)s_2$
for $\lambda \in (0,1)$ implies $s = s_1 = s_2.$ By $\partial_e
{\mathcal S}(E)$ we denote the set of all extremal states of
${\mathcal S}(E)$. We say that a net of states, $\{s_\alpha\}$, on
$E$ {\it weakly converges} to a state $s$ on $E$ if $s_\alpha(a) \to
s(a)$ for any $a \in E$. In this topology, ${\mathcal S}(E)$ is a
compact Hausdorff topological space and every state on $E$ belongs
to  the weak closure of the convex hull of the extremal states, as
the Krein-Mil'man Theorem states, see e.g. \cite[Thm 5.17]{Goo}.

Let $(G,u)$ be an Abelian po-group with strong unit. By a {\it
state} on $(G,u)$ we mean any mapping $s:\ G \to \mathbb R$ such
that (i) $s(g+h) = s(g)+ s(h)$ for all $g,h \in G,$ (ii) $s(G^+)
\subseteq \mathbb R^+$, and (iii) $s(u) = 1.$ In other words, a
state on $(G,u)$ is any po-group homomorphism from $(G,u)$ into the
po-group $(\mathbb R,1)$ that preserves fixed strong units. We
denote by ${\mathcal S}(G,u)$ the set of all states and by
$\partial_e {\mathcal S}(G,u)$ the set  of all extremal states,
respectively,  on $(G,u)$. According to \cite[Cor. 4.4]{Goo},
${\mathcal S}(G,u) \ne \emptyset$ whenever $u >0.$ In a similar way
as for effect algebras, we define the weak convergence of states on
$(G,u)$, and analogously, ${\mathcal S}(G,u)$ is a compact Hausdorff
topological space and every state on $(G,u)$ is a weak limit of a
net of convex combinations from $\partial_e {\mathcal S}(G,u).$

If $E= \Gamma(G,u)$, where $(G,u)$ is an interpolation po-group or a
unigroup, then due to the categorical equivalence, every state on
$E$ can be extended to a unique state on $(G,u)$, and conversely,
the restriction of any state on $(G,u)$ to $E$ gives a state on $E$.
Therefore, if $E$ satisfies (RDP) or more generally if $E$ is an
interval effect algebra, ${\mathcal S}(E) \ne \emptyset$. In
addition, extremal states on $E$ are the restrictions of extremal
states on $(G,u)$, the space ${\mathcal S}(E)$ is affinely
homeomorphic with ${\mathcal S}(G,u)$, and the space $\partial_e
{\mathcal S}(E)$ is homeomorphic with $\partial_e {\mathcal
S}(G,u)$. We recall that a mapping from one convex set onto another
convex set is {\it affine} if it preserves convex combinations.

We recall that if $E$ is an MV-algebra, then a state $s$ on $E$ is
extremal if and only if $s(x\wedge y) =\min\{s(x),s(y)\}$, $x,y \in
E$, therefore $\partial_e {\mathcal S}(E)$ is compact (see, e.g.
\cite[Prop. 6.1.19]{DvPu},  \cite[Thm 12.18]{Goo}).

On the other hand,  if $E$ is an effect algebra with (RDP), then
$\partial_e {\mathcal S}(E)$ is not necessarily compact,   see for
example \cite[Exam. 6.10]{Goo}.

We say that a system of states, ${\mathcal S},$ on an effect algebra
$E$ is  {\it order determining} if $s(a)\le s(b)$ for any $s \in
{\mathcal S}$ yields $a\le b.$

Let $H$ be a separable Hilbert space (real, complex or quaternionic)
with an inner product $(\cdot,\cdot)$, and ${\mathcal L}(H)$ be the
system of all closed subspaces of $H.$ Then ${\mathcal L}(H)$ is a
complete orthomodular lattice \cite{Dvu}. Given a unit vector $\phi
\in H$, let
$$
p_\phi(M):= (P_M\phi,\phi),\quad M \in {\mathcal L}(H),
$$
where $P_M$ is the orthogonal projector of $H$ onto $M.$ Then
$p_\phi$ is a $\sigma$-additive state on ${\mathcal L}(H),$ called a
pure state. The system of all pure states is order determining.  If
$T$ is a positive Hermitian operator of finite trace (i.e.
$\sum_i(T\phi_i,\phi_i) <\infty$ for any orthonormal basis
$\{\phi_i\}$ of $H$, and we define $\mbox{tr}(T)
:=\sum_i(T\phi_i,\phi_i)$), then
$$
s_T(M):=\mbox{tr}(TP_M),\quad M \in {\mathcal L}(H),\eqno(3.3)
$$
is a $\sigma$-additive state, and according to famous Gleason's
Theorem, \cite{Gle}, \cite[Thm 3.2.24]{Dvu}, if $\dim H\ge 3,$ for
every $\sigma$-additive state $s$ on ${\mathcal L}(H)$, there is a
unique positive Hermitian operator $T$ of finite trace such that
$s=s_T.$

If now $s$ is a finitely additive state on ${\mathcal L}(H),$ then
due to the Aarnes Theorem, \cite[Thm 3.2.28]{Dvu}, $s$ can be
uniquely decomposed  to the form
$$ s = \lambda s_1 +(1-\lambda)s_2,$$
where $s_1$ is a $\sigma$-additive state and $s_2$ is a finitely
additive state vanishing on each finite-dimensional subspace of $H.$

Moreover, a pure state $p_\phi$ is an extreme point of the set of
$\sigma$-additive states, as well as it is an extremal state of
${\mathcal L}(H).$ Other extremal states of ${\mathcal L}(H)$ vanish
on each finite-dimensional subspace of $H.$  Since every state on
${\mathcal L}(H)$ can be  extended into a unique state on ${\mathcal
B}(H),$ see e.g. \cite[Thm 3.3.1]{Dvu},  the state spaces ${\mathcal
S}({\mathcal L}(H)),$ ${\mathcal S}({\mathcal E}(H))$ and ${\mathcal
S}({\mathcal B}(H))$ are mutually affinely homeomorphic whenever
$3\le \dim H\le \aleph_0.$

It is worthy to recall that ${\mathcal L}(H)$ is isomorphic to the
set $\mathcal P(H)$ of all orthogonal projectors of $H.$

We say that a po-group $G$ is {\it Archimedean} if for $x,y \in G$
such that $nx \le y$ for all positive integers $n \ge 1$, then $x
\le 0.$  It is possible to show that a unital po-group $(G,u)$  is
Archimedean iff $G^+=\{g \in G:\ s(g) \ge 0$ for all $ s \in
{\mathcal S}(G,u)\},$  \cite[Thm 4.14]{Goo}.  We have the following
characterization of interval effect algebras coming from Archimedean
po-groups.

\begin{proposition}\label{pr:4.2}  Let $E=\Gamma(G,u),$ where is
$(G,u)$ is a unigroup. The following statements are equivalent:

\begin{enumerate}
\item[{\rm (i)}]   ${\mathcal S}(E)$ is order determining.

\item[{\rm (ii)}] $E$ is isomorphic to some effect-clan.

\item[{\rm (iii)}] $G$ is Archimedean.

\end{enumerate}

\end{proposition}

\begin{proof} (i) $\Rightarrow$ (ii). Given $a\in E$, let $\hat a$ be a
function from ${\mathcal S}(E)$ into the real interval $[0,1]$ such
that $\hat a(s):=s(a)$ for any $s \in {\mathcal S}(E),$ and let
${\hat E}=\{\hat a:\ a \in E\}.$ We endow $\hat E$ with pointwise
addition, so that $\hat E$ is an effect-clan. Since ${\mathcal
S}(E)$ is order determining, the mapping $a\mapsto \hat a$ is an
isomorphism and $\hat E$ is Archimedean.

(ii) $\Rightarrow$ (iii).  Let $\hat E$ be any effect-clan
isomorphic with $E,$ and let $a\mapsto \hat a$ be such an
isomorphism. Let $G(\hat E)$ be the po-group generated by $\hat E.$
Then $G(\hat E)$ consists of all functions of the form $\hat
a_1+\cdots+\hat a_n -\hat b_1-\cdots- \hat b_m,$ and $\hat 1$ is its
strong unit.   Due to the categorical equivalence, $(G,u)$ and
$(G(\hat E),\hat 1)$ are isomorphic. But $(G(\hat E),\hat 1)$ is
Archimedean, so is $(G,u).$

(iii) $\Rightarrow$ (i). Due to \cite[Thm 4.14]{Goo}, $G^+=\{g \in
G:\ s(g) \ge 0$ for all $ s \in {\mathcal S}(G,u)\},$ which means
that ${\mathcal S}(G,u)$ is order determining. Hence, the
restrictions of all states on $(G,u)$ onto $E=\Gamma(G,u)$ entail
${\mathcal S}(E)$ is order determining.
\end{proof}

\section{Affine Continuous Functions}

We present some important notions about affine continuous functions
that are important for theory of effect algebras. A good source for
this topic are monographs \cite{Goo, Alf}.

Let $K$ be a convex subset of a real vector space $V.$ A point $x\in
K$ is said to be {\it extreme} if from $x= \lambda
x_1+(1-\lambda)x_2,$ where $x_1,x_2 \in K$ and $0<\lambda <1$ we
have $x=x_1=x_2.$ By $\partial_e K$ we denote the set of extreme
points of $K,$ and if $K$ is a compact convex subset of a locally
convex Hausdorff topological space, then due to the Krein-Mil'man
Theorem \cite[Thm 5.17]{Goo}, $K$ equals to the closure of the
convex hull of $\partial_e K,$ i.e. $K = (\mbox{con}(\partial_e
K))^-,$ where $^-$ and $\mbox{con}$ denote the closure and the
convex hull, respectively.

A mapping $f:\ K \to \mathbb R$ is said to be {\it affine} if, for
all $x,y \in K$ and any $\lambda \in [0,1]$, we have $f(\lambda x
+(1-\lambda )y) = \lambda f(x) +(1-\lambda ) f(y)$.

Given a compact convex set $K$ in a topological vector space, we
denote by $\mbox{Aff}(K)$ the set of all affine continuous functions
on $K$. Then $\mbox{Aff}(K)$ is a po-subgroup of the po-group
$\mbox{C}(K)$ of all continuous real-valued functions on $K$ (we
recall that, for $f,g \in \mbox{C}(K),$ $f \le g$ iff $f(x)\le g(x)$
for any $x \in K$), hence it is an Archimedean unital po-group with
the strong unit $1$. In addition, $\mbox{C}(K)$ is  an $\ell$-group
(= lattice ordered group).

For example, if $E$ is an effect algebra such that its state space
is nonempty, then for any $a \in E$, the function $\hat a(s) =s(a),$
$s \in {\mathcal S}(E),$ is a continuous affine function belonging
to $\mbox{Aff}({\mathcal S}(E)).$

In contrast to $\mbox{C}(K),$  the space $\mbox{Aff}(K)$ is not
necessarily a lattice. Indeed, according to \cite[Thm 11.21]{Goo},
it is possible to show that if $(G,u)$ is a unital interpolation
group, then $\mbox{Aff}({\mathcal S}(G,u))$ is a lattice if and only
if $\partial_e {\mathcal S}(G,u)$ is compact.

Let now $K$ be a compact convex subset of a locally convex Hausdorff
space, and let $S = {\mathcal S}(\mbox{Aff}(K),1).$ Then the
evaluation mapping $\psi:\ K \to S$ defined by $\psi(x)(f)=f(x)$ for
all $f \in \mbox{Aff}(K)$ $(x \in K)$ is an affine homeomorphism of
$K$ onto $S,$ see \cite[Thm 7.1]{Goo}.

Let ${\mathcal B}_0(K)$ be the Baire $\sigma$-algebra, i.e. the
$\sigma$-algebra generated by compact $G_\delta$-sets (a
$G_\delta$-set is a countable intersection of open sets) of a
compact set $K$, and the elements of ${\mathcal B}_0(K)$ are said to
be {\it Baire sets}, and a $\sigma$-additive (signed) measure of
${\mathcal B}_0(K)$ is a {\it Baire measure}.

Similarly, ${\mathcal B}(K)$ is the Borel $\sigma$-algebra of $K$
generated by all open subsets of $K$ and any element of ${\mathcal
B}(K)$ is said to be a {\it Borel set}, and any $\sigma$-additive
(signed) measure on it is said to be a {\it Borel measure}. It is
well-known that each Baire measure can be extended to a unique
regular Borel measure.

Let ${\mathcal M}(K)$  denote the set of all finite signed regular
Borel measures  on ${\mathcal B}(K)$ and let  ${\mathcal M}_1^+(K)$
denote the set of all probability measures, that is, all positive
regular Borel measures $\mu \in {\mathcal M}(K)$ such that $\mu(K) =
1.$ We recall that a Borel measure $\mu$ is called regular if

$$\inf\{\mu(O):\ Y \subseteq O,\ O\ \mbox{open}\}=\mu(Y)
=\sup\{\mu(C):\ C \subseteq Y,\ C\ \mbox{closed}\}
$$
for any $Y \in {\mathcal B}(K).$

Due to the Riesz Representation  Theorem, \cite{DuSc},  we can
identify the set ${\mathcal M}(K)$ with the dual of the space
$\mbox{C}(K)$ by setting
$$
\mu(f) = \int_K f \dx\mu \eqno(4.1)
$$
for all $\mu \in {\mathcal M}(K)$ and all $f \in \mbox{C}(K).$

We endow ${\mathcal M}(K)$ with the weak$^*$  topology, i.e., a net
$\mu_\alpha$ converges to an element $\mu \in {\mathcal M}(K)$ iff
$\mu_\alpha(f) \to \mu(f)$ for all $f \in \mbox{C}(K)$.

Let $x$ be a fixed point in $K.$ Then there is a unique Borel
regular measure $\mu = \mu_x$ (depending on $x$) such that

$$f(x)= \int_K f \dx\mu, \ f \in \mbox{C}(K). \eqno(4.2)$$

Let $\delta_x$ be the Dirac measure concentrated at the point $x \in
K,$ i.e., $\delta_x(Y)= 1$ iff $x \in Y$, otherwise $\delta_x(Y)=0;$
then every Dirac measure is a regular Borel probability measure. In
formula (4.2) we have then $\mu_x = \delta_x.$  Moreover, \cite[Prop
5.24]{Goo},  the mapping
$$
\epsilon: x\mapsto \delta_x \eqno(4.3)
$$
gives a homeomorphism of $K$ onto $\partial_e {\mathcal M}_1^+(K).$

Suppose that $K$ is a nonempty compact subset of a locally convex
space $V,$ and let $\mu \in {\mathcal M}_1^+(K).$ A point $x$ in $V$
is said to be {\it represented} by $\mu$ (or $x$ is the {\it
barycenter} of $\mu$) if
$$
f(x) = \int_K f \dx\mu,\quad f \in V^*,
$$
where $V^*$ stands for the set of continuous linear functionals $f$
in $V.$

From the Krein--Mil'man Theorem, it is possible to derive \cite[p.
6]{Phe}, that  every point $x$ of a compact convex set $K$ of a
locally convex space $V$  is the barycenter of a probability measure
$\mu \in {\mathcal M}_1^+(K)$ such that $\mu((\partial_e K)^-)=1.$

We note that according to a theorem by Bauer \cite[Prop 1.4]{Phe},
if $K$ is a convex compact subset of a locally convex space $V,$
then a point $x $ is an extreme point of $K$ iff the Dirac measure
$\delta_x$ is a unique probability measure from ${\mathcal
M}_1^+(K)$ whose barycenter is $x.$

\section{Choquet Simplices}

We show that, for theory of effect algebras, simplices are
important, for more about simplices see in \cite{Goo, Alf, Phe}. In
what follows, for the reader convenience, we present some necessary
notions and results on Choquet and Bauer simplices.

  We recall that a {\it convex cone} in a real linear
space $V$ is any subset $C$ of  $V$ such that (i) $0\in C,$ (ii) if
$x_1,x_2 \in C,$ then $\alpha_1x_1 +\alpha_2 x_2 \in C$ for any
$\alpha_1,\alpha_2 \in \mathbb R^+.$  A {\it strict cone} is any
convex cone $C$ such that $C\cap -C =\{0\},$ where $-C=\{-x:\ x \in
C\}.$ A {\it base} for a convex cone $C$ is any convex subset $K$ of
$C$ such that every non-zero element $y \in C$ may be uniquely
expressed in the form $y = \alpha x$ for some $\alpha \in \mathbb
R^+$ and some $x \in K.$

We recall that in view of \cite[Prop 10.2]{Goo}, if $K$ is a
non-void convex subset of $V,$ and if we set

$$ C =\{\alpha x:\ \alpha \in \mathbb R^+,\ x \in K\},$$
then $C$ is a convex cone in $V,$ and $K$ is base of $C$ iff there
is a linear functional $f$ on $V$ such that $f(K) = 1$ iff $K$ is
contained in a hyperplane in $V$ which misses the origin.

Any strict cone $C$ of $V$ defines a partial order $\le_C$ via $x
\le_C y$ iff $y-x \in C.$ It is clear that $C=\{x \in V:\ 0 \le_C
x\}.$ A {\it lattice cone} is any strict convex cone $C$ in $V$ such
that $C$ is a lattice under $\le_C.$

A {\it simplex} in a linear space $V$ is any convex subset $K$ of
$V$ that is affinely isomorphic to a base for a lattice cone in some
real linear space. A  simplex $K$ in a locally convex Hausdorff
space is said to be (i) {\it Choquet} if $K$ is compact, and (ii)
{\it Bauer} if $K$ and $\partial_e K$ are compact.

For example, for the important quantum mechanical  example, if $H$
is a separable complex Hilbert space, ${\mathcal S}({\mathcal
E}(H))$ is not a simplex due to \cite[Thm 4.4]{AlSc}, where it is
shown that the state space of ${\mathcal S}({\mathcal E}(H))$ for
the two-dimensional   $H$ is isomorphic to the unit
three-dimensional real ball, consequently ${\mathcal S}({\mathcal
E}(H))$ is not a simplex whenever $\dim H>1.$ Of course, ${\mathcal
E}(H)$ does not satisfy (RDP):

\begin{theorem}\label{th:6.1} If an effect algebra $E$ satisfies
{\rm (RDP)}, then ${\mathcal S}(E)$ is a Choquet simplex.
\end{theorem}

\begin{proof} In view of Theorem \ref{th:2.1}, there is a unital po-group
$(G,u)$ with interpolation such that $E=\Gamma(G,u).$ Since
${\mathcal S}(E)$ is affinely homeomorphic with ${\mathcal S}(G,u)$,
we have that ${\mathcal S}(E)$ is a Choquet simplex iff so is
${\mathcal S}(G,u).$  By \cite[Thm 10.17]{Goo}, ${\mathcal S}(G,u)$
is a Choquet simplex.  We recall that ${\mathcal S}(G,u)$ is a base
for the positive cone of the Dedekind complete lattice-ordered real
vector space of all relatively bounded homomorphisms from $G$ to
$\mathbb R.$ \end{proof}

\begin{example}\label{ex:6.2}
If $E$ is an interval effect algebra that does not satisfies {\rm (RDP)},
then $\mathcal S(E)$ is not necessarily a Choquet simplex.

{\rm Indeed, let $G=\{(x_1,x_2,x_3,x_4)\in \mathbb Z^4:
x_1+x_2=x_3+x_4\},$ $u=(1,1,1,1),$ and $E=\Gamma(G,u).$ If
$s_i(x_1,x_2,x_3,x_4)= x_i,$ then $\partial_e \mathcal
S(E)=\{s_1,s_2,s_3,s_4\}$ and $\mathcal S(E)$ is not a simplex
because $(s_1+s_2)/2 = (s_3+s_4)/2,$ see also \cite[Cor 10.8]{Goo}.

Another example is $\mathcal {S(E(}H))$ for the effect algebra
$\mathcal E(H).$}

\end{example}

The importance of  Choquet and Bauer simplices  follows from the
fact that if $K$ is a convex compact subset of a locally convex
Hausdorff space, then $K$ is a Choquet simplex iff
$(\mbox{Aff}(K),1)$ is an interpolation po-group, \cite[Thm
11.4]{Goo}, and $K$ is a Bauer simplex iff $(\mbox{Aff}(K),1)$ is an
$\ell$-group, \cite[Thm 11.21]{Goo}. Consequently, there is no
MV-algebra whose state space is affinely isomorphic to the closed
square or the closed unit circle.

For two measures $\mu$ and $\lambda$ we write
$$\mu \sim \lambda\quad  \mbox{iff}\quad
\mu(f)=\lambda(f),\ f \in \mbox{Aff}(K).
$$

If  $\mu $ and $\lambda$ are nonnegative regular Borel measures on
$K,$ write
$$
\mu \prec \lambda  \quad  \mbox{iff} \quad \mu(f)\le \lambda(f), \ f
\in \mbox{Con}(K),
$$  where $\mbox{Con}(K)$ is the set of all continuous convex
functions $f$ on $K$ (that is $f(\alpha x_1+(1-\alpha) x_2)\le
\alpha f(x_1)+(1-\alpha)f(x_2)$ for $x_1,x_2\in K$ and $\alpha \in
[0,1]$). Then $\prec$ is a partial order on the cone of nonnegative
measures. The fact $\lambda \prec \mu$ and $\mu \prec \lambda$
implies $\lambda = \mu$ follows from the fact that
$\mbox{Con}(K)-\mbox{Con}(K)$ is dense in $\mbox{C}(K).$

Moreover, for any probability measure (= regular Borel probability
measure) $\lambda$ there is a maximal probability measure $\mu$ such
that $\mu \succ \lambda,$ \cite[Lem 4.1]{Phe}, and every point $x\in
K$ can be represented by a maximal probability measure.

From the Bishop-de Leeuw Theorem, \cite[Cor I.4.12]{Alf}, see also
\cite[p. 24]{Phe}, it follows that if $\mu$ is a maximal probability
measure on ${\mathcal B}(K),$ then for any Baire set $B$ disjoint
with $\partial_e K$, $\mu(B)=0.$  In general, the statement is not
true if we suppose that $B$ is Borel instead Baire.

An important characterization of Choquet simplices  is given by
Choquet--Meyer, \cite[Thm p. 66]{Phe}, saying that a compact convex
subset $K$ of a locally convex space is a Choquet simplex iff, for
every point $x \in K,$ there is a unique maximal probability measure
$\mu_x$ such that $\mu_x \sim \delta_x.$ That is,
$$
f(x)= \int_K f\dx\mu_x, \quad f \in \mbox{Aff}(K).\eqno(5.1)
$$

And the characterization by Bauer, \cite[Thm II.4.1]{Alf}, says that
a Choquet simplex $K$ is Bauer iff for any point $x \in K,$ there is
a unique Baire probability (equivalently, regular Borel) measure
$\mu_x$ such that $\mu_x(\overline{\partial_e K})=1$ and (5.1)
holds.

\section{Effect Algebras with the Bauer State Property}

We present the main results of the paper, namely we show that every
state on an interval effect algebra is an integral. We show that in
some cases every state corresponds to a unique regular Borel
probability measure.  In addition, we generalize Horn--Tarski result
also for interval effect algebras and its subalgebras.

We say that an effect algebra has the {\it Bauer state property}
((BSP), for short) if the state space ${\mathcal S}(E)$ is a Bauer
simplex. For example, every lattice ordered effect algebra with
(RDP) has (BSP).  The next example is a case when also a non-lattice
ordered effect algebra with (RDP) has (BSP).  The next example is
from \cite[Ex 4.20]{Goo}.

\begin{example}\label{ex:7.1} {\rm  Let $H=\mathbb Z$ be a po-group
with the discrete ordering and $G=\mathbb Q \lex H$ be the
lexicographic product. Then $G$ has interpolation but $G$ is not a
lattice. If we set $u=(1,0),$ then $u$ is a strong unit in $G,$ and
let $E=\Gamma(G,u).$ If $s$ is a state on $G,$ then due to
$n(k/n,0)=u$ we have $s(k/n,0)=k/n.$ On the other hand, $-u\le
n(0,b) \le u$ so that $-1\le ns(0,b)\le 1$ which proves $s(0,b)=0.$
Hence, $s(a,b)=a$ proving that ${\mathcal S}(G,u)$ is a singleton
and therefore, $E$ has (BSP). In addition, the direct product $E^n$
has also (BSP) for any integer $n\ge 1.$

We note that in our example  $\mbox{Ker}(s)=\{(0,b):\ b \in H\},$
and then the quotient effect algebra
$(G/\mbox{Ker}(s),u/\mbox{Ker}(s))\cong (\mathbb Q,1).$}
\end{example}

\begin{theorem}\label{th:7.2'}
Let $E$ be an interval effect algebra and let $s$ be a state on $E.$
Let $\psi: E \to \Aff(\mathcal S(E))$ be defined by $\psi(a) := \hat
a,$ $a\in E,$ where $\hat a$ is a mapping from $\mathcal S(E)$ into
$[0,1]$ such that $\hat a(s):=s(a),$ $s \in \mathcal S(E).$  Then
there is a unique state $\tilde s$ on the unital po-group
$(\Aff(\mathcal S(E)),1)$ such that $\tilde s(\hat a) = s(a)$ for
any $a \in E.$  Moreover, this state can be extended to a state  on
$(\mbox{C}(\mathcal S(E),1)).$

The mapping $s \mapsto \tilde s$ defines an affine homeomorphism
from the state space $\mathcal S(E)$ onto $\mathcal
S(\Gamma(\Aff(\mathcal S(E)),1)).$

If, in addition,  $E$ is has an order system of states, then
$\widehat E :=\{\hat a: a \in E\}$ is an effect-clan that is a
sub-effect algebra of $\Gamma(\Aff(\mathcal S(E)),1)$ and every
state $s$ on $E$ can be  extended to a state $\tilde s$ on
$\Gamma(\mbox{C}(\mathcal S(E)),1)$ as well as to a  state on the
unital po-group $(\mbox{C}(\mathcal S(E)),1).$

\end{theorem}

\begin{proof}
Since $E$ is an interval effect algebra and $E=\Gamma(G,u)$  for
some unigroup $(G,u),$ $\mathcal S(E)$ is non-void.  The mapping
$\psi$ can be uniquely extended to a po-group homomorphism $\hat
\psi:G \to \Aff(\mathcal S(E)).$  Let $\hat s$ be a state on $(G,u)$
that is a unique extension of a state $s.$ Applying now \cite[Prop
7.20]{Goo}, the statement follows, in particular, we have that there
is a unique state $\tilde s$  on $(\Aff(\mathcal S(E)),1)$ such that
$\tilde s(\hat a) = s(a),$ $a\in E.$

Due to \cite[Cor 4.3]{Goo}, we see that this state can be extended
to a state on $(\mbox{C}(\mathcal S(E)),1)$ as well as to a state on
$\Gamma(\mbox{C}(\mathcal S(E)),1).$

The affine homeomorphism $s \mapsto \tilde s$ follows from \cite[Thm
7.1]{Goo}.

If $E$ is Archimedean, the result follows from Proposition
\ref{pr:4.2} and the first part of the present proof.
\end{proof}

\begin{theorem}\label{th:7.3'}  Let $E$ be an interval effect algebra such that
$\mathcal S(E)$ is a simplex and let $s$ be a state on $E.$ Then
there is a unique maximal regular Borel probability measure $\mu_s
\sim \delta_s$ on $\mathcal B(\mathcal S(E))$ such that

$$ s(a) = \int_{\mathcal S(E)} \hat a(x) \dx \mu_s(x),\quad a \in
E. \eqno(6.1)$$
\end{theorem}

\begin{proof}
Due to our hypothesis, $\mathcal S(E)$ is a Choquet simplex.  By
Theorem \ref{th:7.2'}, there is a unique state $\tilde s$ on
$(\Aff(\mathcal S(E)),1)$ such that $\tilde s(\hat a) = s(a),$ $a
\in A.$

Applying the result by Choquet--Meyer, \cite[Thm p. 66]{Phe}, see
(5.1), we have

$$f(s)=\int_{\mathcal S(E)} f(x) \dx\mu_s, \quad f \in \Aff(\mathcal
S(E)).
$$
Since $\hat a \in \Aff(\mathcal S(E))$ for any $a\in E,$ we have the
statement in question.
\end{proof}

We recall that a state $s$ on $E$ is $\sigma$ additive if
$a_n\nearrow a,$ then $\lim_n s(a_n) = s(a).$ In the same manner we
define a $\sigma$-additive state on $(G,u).$

\begin{theorem}\label{th:7.4'}
Let $E=\Gamma(G,u)$ be an interval effect algebra where $(G,u)$ is a
unigroup, and let $\mathcal S(E)$ be a simplex. If $s$ is
$\sigma$-additive, then its unique extension, $\hat s,$ on $(G,u)$
is $\sigma$-additive.
\end{theorem}

\begin{proof} Let $s$ be a $\sigma$-additive state on $E$ and let
$\hat  s$ be its unique extension onto $(G,u)$. By Theorem
\ref{th:7.3'}, there is a unique maximal regular Borel probability
measure $\mu_s \sim \delta_s$ such that (6.1) holds.

For any $g \in G$, let $\hat g$ denote the function on $\mathcal
S(E)$ such that $\hat g(\hat s) = \hat s(g),$ $ s \in \mathcal
S(E),$ where $\hat s$ is a unique extension of $s$ onto $(G,u).$ If
$g \in G^+,$ then $g= g_1+\cdots +g_k$ for some $g_1,\ldots,g_k \in
E,$ $\hat g = \hat g_1+\cdots +\hat g_k,$ and

\begin{eqnarray*}
\hat s(g) &=& s(g_1)+\cdots+ s(g_k)\\
&=& \int_{\mathcal S(E)} \hat g_1(x) \dx \mu_s(x) +\cdots +
\int_{\mathcal S(E)} \hat g_k(x) \dx \mu_s(x)\\
&=& \int_{\mathcal S(E)} \hat g(x) \dx \mu_s(x).
\end{eqnarray*}

The $\sigma$-additivity of $\mu_s$ entails due to the Lebesgue
Convergence Theorem, \cite{Hal}, that  if $g_n \nearrow g,$ then
$\lim_n \hat s(g_n) = \hat s(g)$ for any $g \in G^+.$

If now $g_n \searrow 0$, then $\lim_n \hat s(g_n)=0;$ in fact, there
is an integer $k\ge 1$ such that $g_n \le g_1\le ku$ for any $n \ge
1.$ This gives $ku - g_n \nearrow ku$ and $\lim_n \hat s(ku-g_n) =
\hat s (ku),$ so that $\lim_n \hat s(g_n) =0.$

Finally, let $g\in G$ be arbitrary and let $g_n \nearrow g$. Then
$g-g_n \searrow 0$ that gives $\lim_n \hat s(g_n)= \hat s(g)$ and
$\hat s$ is $\sigma$-additive on $(G,u)$ as claimed.
\end{proof}

\begin{theorem}\label{th:7.5'}  Let $E$ be an interval effect algebra with
{\rm (BSP)} and let $s$ be a state on $E.$ Then there is a unique
regular Borel probability measure, $\mu_s,$ on $\mathcal B(\mathcal
S(E))$ such that

$$ s(a) = \int_{\partial_e \mathcal S(E)} \hat a(x) \dx \mu_s(x),\quad a \in
E. \eqno(6.2)$$
\end{theorem}

\begin{proof}  Due to Theorem \ref{th:7.3'}, we have a unique regular
Borel probability measure $\mu_s\sim \delta_s$ such that (6.1)
holds. The characterization of Bauer simplices, \cite[Thm
II.4.1]{Alf}, says that then $\mu_s$ is a unique regular Borel
probability measure $\mu_s$ on such that (6.1) holds and
$\mu_s(\partial_e \mathcal S(E)) = 1.$ Hence, (6.2) holds.
\end{proof}

We recall that Theorems \ref{th:7.2'}--\ref{th:7.5'} are applicable
to every (i) effect algebra  satisfying (RDP) and (BSP),
respectively, (ii) MV-algebra and therefore, Theorem \ref{th:7.5'}
generalizes the result of \cite{Pan, Kro}.  Since the state space of
$\mathcal E(H)$ is not a simplex, the above theorems are not
applicable for this important case.  However, the next result shows
that also in this case, every state $s$ on $\mathcal E(H)$ is an
integral over some Bauer simplex, but the uniqueness of a regular
Borel probability measure $\mu_s$ is not more guaranteed.

\begin{theorem}\label{th:7.7'}  Let $s$ be a state  on an interval effect
algebra $E.$ Then there is a regular Borel probability measure,
$\mu_s,$ on the Borel $\sigma$-algebra $\mathcal B(\mathcal S(E))$
such that

$$s(a)=\int_{\mathcal S(E)}\hat a(x) \dx \mu_s(x),\quad a \in
E. \eqno(6.3)
$$

\end{theorem}

\begin{proof}
The set $\mathcal S(E)$ is non-void and a compact convex Hausdorff
topological space. Given a state $s\in \mathcal S(E),$ by Theorem
\ref{th:7.2'}, $s$ can be uniquely extended to a state $\tilde s$ on
$(\Aff(S(E)),1).$  Since $(\Aff(\mathcal S(E)),1))$ is a unital
po-group that is a subgroup of the unital group $(\mbox{C}(\mathcal
S(E)),1),$ but $(\mbox{C}(\mathcal S(E)),1)$ is an $\ell$-group.
Therefore, this state $\tilde s$ can be by \cite[Cor 4.3]{Goo}
extended to a state $s'$ on the unital $\ell$-group
$\mbox{C}(\mathcal S(E),1).$ But $(\mbox{C}(\mathcal S(E)),1)$
satisfies interpolation and $\Gamma(\mbox{C}(\mathcal S(E)),1)$ has
(BSP), therefore by Theorem \ref{th:7.5'}, there is a regular Borel
probability measure $\nu_s$ on the Bauer simplex $\Omega = \mathcal
S(\mbox{C}(\mathcal S(E)),1)$ such that

$$f(s') = \int_{\partial_e \Omega} f(x) \dx \nu_s(x),\quad f \in
\Gamma(\mbox{C}(\mathcal S(E)),1).\eqno(6.4)
$$

Let $\epsilon: \mathcal S(E) \to \partial_e \mathcal M^+_1(\mathcal
S(E))$ defined by (4.3). Then $\epsilon$ is a homeomorphism. For any
$a \in E$, by \cite[Cor 7.5]{Goo}, there is a unique continuous
affine mapping $\tilde a: \mathcal M^+_1(\mathcal S(E)) \to [0,1]$
such that $\hat a = \tilde a \circ \epsilon.$  Therefore, $\tilde
a\in \Gamma(\mbox{C}(\mathcal S(E)),1)$ and $s(a)= \hat a(s) =
\tilde a(\delta_s).$ The homeomorphism $\epsilon$ defines a regular
Borel measure $\mu_s:= \nu_s\circ \epsilon$ on $\mathcal {B(S}(E)),$
and using (6.4) and transformation of integrals, we have

\begin{eqnarray*}
s(a)=\tilde a(\epsilon (s)) &=& \int_{\partial_e \Omega} \tilde a(y)
\dx \nu_s(y)= \int_{\epsilon(\mathcal S(E))} \hat a\circ
\epsilon^{-1}(y) \dx\nu_s(y)\\
&=& \int_{\mathcal S(E)} \hat a(x) \dx \mu_s(x).
\end{eqnarray*}
This implies (6.3).
\end{proof}

\begin{remark}\label{re:1}
{\rm We note that  reasons why we do not know to guarantee the
uniqueness of $\mu_s$ in (6.3)  are facts that $\mathcal S(E)$ is
not necessarily a Bauer simplex nor a Choquet one, and
$\Aff(\mathcal S(E))$ is sup-norm closed in $\mbox{C}(\mathcal
S(E))$ so we cannot extend a state on $(\Aff(\mathcal S(E)),1)$ by
continuity to a state on $(\mbox{C}(\mathcal S(E)), 1).$}
\end{remark}

\begin{corollary}\label{co:7.6'} Let $E$ be an interval effect algebra with
{\rm (BSP)}.  The mapping $s \mapsto  \mu_s,$ where $\mu_s$ is a
unique regular Borel probability measure satisfying {\rm (6.2)}, is
an affine homeomorphism between $\mathcal S(E)$ and the set of
regular Borel probability measures on ${\mathcal B(S}(E))$ endowed
with the weak$^*$ topology. A state $s$ on $E$ is extremal if and
only if $\mu_s$ in {\rm (6.2)} is extremal. In such a case, $\mu_s =
\delta_s.$
\end{corollary}

\begin{proof}  The fact that that the mapping $s\mapsto \mu_s$ is
affine and injective follows from Theorem \ref{th:7.5'} and (6.2).
The surjectivity is clear because if $\mu$ is a regular Borel
probability measure, then  $\mu $ defines via (6.2) some state, $s,$
on $E.$ The continuity follows from \cite[Thm II.4.1(iii)]{Alf}.

Since every Dirac measure $\delta_s$ is always a regular Borel
probability measure, (6.2) entails that $s$ has to be extremal.
Conversely, if $s$ is extremal, the uniqueness of $\mu_s$ yields
that $\mu_s=\delta_s.$
\end{proof}

It is worthy to recall that Corollary \ref{co:7.6'} shows that every
(finitely additive) state on an interval effect algebra satisfying
(BSP)  is in a one-to-one correspondence with ($\sigma$-additive)
regular Borel probability measure on some Borel $\sigma$-algebra
trough formula (6.2). This observation is interesting because
according to de Finetti, a probability measure  is only  a finitely
additive measure, whilst by Kolmogorov \cite{Kol} a probability
measure is assumed to be $\sigma$-additive. (6.2) shows that there
is a natural coexistence between both approaches.

\vspace{2mm}

A famous result by Horn and Tarski \cite{HoTa} states that every
state on a  Boolean subalgebra $A$ of a Boolean algebra $B$ can be
extended to a state on $B,$ of course not in unique way. E.g. the
two-element Boolean subalgebra, $\{0,1\}$ has a unique state,
0-1-valued,  and it can be extended to many distinct  states on $B.$
This result was generalized for MV-algebras, see \cite[Thm 6]{Kro1}.
In what follows, we generalize this result also for interval effect
algebras, and this gives a variant of the Horn--Tarski Theorem.

\begin{theorem}\label{th:7.7}  Let $E$ be an interval effect algebra and $F$
its sub-effect algebra.  Then every state $s$ on $F$ can be extended
to a state on $E.$
\end{theorem}

\begin{proof}  Let $E$ be an interval effect algebra and $F$ be its
subalgebra. According to \cite[Cor 1.4.5]{DvPu}, $F$ is also an
interval effect algebra.  According to Remark \ref{re:amb}, we can
assume that $E = \Gamma(G,u),$ where $u$ is a universal strong unit
for $E.$ Then $F\cong\Gamma(H',u'),$ where $H'$ is a po-group with a
universal strong unit $u'.$ Let $\iota_F$ be the embedding of $F$
into $E$ and let $\iota$ be an isomorphism of $\Gamma(H',u')$ onto
$F.$  Assume that $s$ is a state on $F.$  Then $s'= s\circ \iota$ is
a state on $\Gamma(H',u').$ Applying again Remark \ref{re:amb}, the
state $s'$ can be extended to a unique state, $s'',$ on
$\Gamma(H',u').$ On the other hand, the mapping $\iota$ can be
extended to a unique group homomorphism, $\widehat \iota,$  from
$(H',u')$ into $(G,u)$ such that $\widehat {\iota}(u')=u.$  Because
$\iota$ was injective, so is $\widehat \iota,$ and therefore,
$\widehat \iota$ is an embedding. This entails that on the unital
subgroup $(\widehat \iota(H'),u)$ we have a state $s'''=s''\circ
(\widehat i)^{-1}$ that is an extension of the original state $s$ on
$F.$ Applying \cite[Cor 4.3]{Goo}, we see that $s'''$ can be
extended to a state $\tilde s$ on $(G,u)$ so that, the restriction
of $\tilde s$ onto $E=\Gamma(G,u)$ is an extension of a state $s$ on
$F.$
\end{proof}

\begin{corollary}\label{co:7.8}  Let $E$ be an effect algebra
satisfying {\rm (RDP)} and  let $F$ be its subalgebra.  Then every
state   on $F$ can be extended to a state on $E.$
\end{corollary}

\begin{proof}  If $E$ satisfies (RDP), then $E$ is an interval
effect algebra.  The desired result now follows from  Theorem
\ref{th:7.7}.
\end{proof}

\section{Conclusion}

Using techniques of simplices and affine continuous functions on a
convex compact Hausdorff space, we have showed that every state on
an interval effect algebra, i.e. on an interval $[0,u]$ of a unital
po-group $(G,u)$,  can be represented by an integral, Theorems
\ref{th:7.3'}--\ref{th:7.7'}. A sufficient condition for this is the
condition that the state space, $\mathcal S(E),$ of an interval
effect algebra $E$ is a Choquet simplex, Theorem \ref{th:7.4'}. In
this case  it is even a Bauer simplex, Theorem \ref{th:7.5'}, and
there is a one-to-one correspondence between the set of all states
on $E$ and the set of all regular Borel probability measures on the
Borel $\sigma$-algebra $\mathcal {B(S}(E)),$ Corollary
\ref{co:7.6'}.

On the other side, an important example for mathematical foundations
of quantum mechanics,  the state space of the effect algebra
$\mathcal E(H),$ the set of all Hermitian operators  on a Hilbert
space $H$ that are between the zero operator and the identity, is
not a simplex, but also in this case we can represent every state as
an integral through a regular Borel probability measure over some
Bauer simplex, Theorem \ref{th:7.7'}, however in this case, the
uniqueness of a Borel probability measure is not more guaranteed.

In addition, we have proved a variant of the Horn--Tarski Theorem
that every state on a subalgebra of an unital effect algebra can be
extended to a state on the interval effect algebra, Theorem
\ref{th:7.7}.

\end{document}